\documentclass{article} 
\usepackage{hyperref,bbm}
\usepackage{url}
\usepackage{color}
\usepackage{amsmath, amsthm, amsfonts, amssymb, mathrsfs,float}
\usepackage{multirow}
\usepackage{mathdots}
\usepackage{algorithm}
\usepackage{algorithmic}
\usepackage{graphicx,graphics}
\usepackage{subfigure}
\usepackage{verbatim}

\newtheorem{theorem}{Theorem}

\newtheorem{corollary}{Corollary}
\newtheorem{proposition}{Proposition}

\newtheorem{fact}{Fact}

\newcommand{\bz}{\mathbf 0}

\newcommand{\R}{\mathbb R}

\newcommand{\C}{\mathbb C}
\newcommand{\T}{\mathbb T}

\newcommand{\ignore}[1]{}

\begin{document}
\title{On the Estimation Performance and Convergence Rate of the Generalized Power Method for Phase Synchronization}
\author{
Huikang Liu\thanks{Department of Systems Engineering and Engineering Management, The Chinese University of Hong Kong,
Shatin, N. T., Hong Kong. E-mail: {\tt hkliu@se.cuhk.edu.hk}}\and
Man-Chung Yue\thanks{Department of Systems Engineering and Engineering Management, The Chinese University of Hong Kong,
Shatin, N. T., Hong Kong. E-mail: {\tt mcyue@se.cuhk.edu.hk}}\and
Anthony Man-Cho So\thanks{Department of Systems Engineering and Engineering Management, and, by courtesy, CUHK-BGI Innovation Institute of Trans-omics, The Chinese University of Hong Kong, Shatin, N. T., Hong Kong. E-mail:
{\tt manchoso@se.cuhk.edu.hk}}
}
\date{\today}
\maketitle
\begin{abstract}
  An estimation problem of fundamental interest is that of phase (or angular) synchronization, in which the goal is to recover a collection of phases (or angles) using noisy measurements of relative phases (or angle offsets).  It is known that in the Gaussian noise setting, the maximum likelihood estimator (MLE) has an expected squared $\ell_2$-estimation error that is on the same order as the Cram\'er-Rao lower bound.  Moreover, even though the MLE is an optimal solution to a non-convex quadratic optimization problem, it can be found with high probability using semidefinite programming (SDP), provided that the noise power is not too large.  In this paper, we study the estimation and convergence performance of a recently-proposed low-complexity alternative to the SDP-based approach, namely, the generalized power method (GPM).  Our contribution is twofold.  First, we bound the rate at which the estimation error decreases in each iteration of the GPM and use this bound to show that all iterates---not just the MLE---achieve an estimation error that is on the same order as the Cram\'er-Rao bound.  Our result holds under the least restrictive assumption on the noise power and gives the best provable bound on the estimation error known to date.  It also implies that one can terminate the GPM at any iteration and still obtain an estimator that has a theoretical guarantee on its estimation error. Second, we show that under the same assumption on the noise power as that for the SDP-based method, the GPM will converge to the MLE at a linear rate with high probability.  This answers a question raised in~\cite{boumal2016nonconvex} and shows that the GPM is competitive in terms of both theoretical guarantees and numerical efficiency with the SDP-based method.  At the heart of our convergence rate analysis is a new error bound for the non-convex quadratic optimization formulation of the phase synchronization problem, which could be of independent interest.  As a by-product, we give an alternative proof of a result in~\cite{boumal2016nonconvex}, which asserts that every second-order critical point of the aforementioned non-convex quadratic optimization formulation is globally optimal in a certain noise regime.

\end{abstract}

\section{Introduction}\label{sec:intro}
The problem of phase synchronization is concerned with the estimation of a collection of phases\footnote{Throughout the paper, the term ``phase'' refers to a complex number with unit modulus.}
based on noisy measurements of the relative phases. Formally, let $z^\star\in \mathbb{T}^n=\{w\in\C^n: |w_1|=\cdots=|w_n|=1\}$ be an unknown phase vector.  Given noisy measurements of the form
\begin{equation} \label{eq:measure}
C_{j\ell} = z_j^\star \bar{z}_\ell^\star +\Delta_{j\ell} \quad\mbox{for } 1\le j<\ell\le n,
\end{equation}
where $\overline{(\cdot)}$ denotes the complex conjugate and $\Delta_{j\ell}\in\C$ is the noise in the measurement of the relative phase $z_j^\star\bar{z}_\ell^\star$, our goal is to find an estimate $\hat{z}\in\T^n$ of $z^\star\in\T^n$ that best fits those measurements in the least-squares sense. In other words, we are interested in solving the following optimization problem:
\begin{equation} \label{eq:AS-ls}
\hat{z} \in \arg\min_{z\in\T^n} \sum_{1\le j<\ell\le n} \left| C_{j\ell} - z_j\bar{z}_\ell \right|^2.
\end{equation}
Despite its simple description, the phase synchronization problem arises in a number of applications, including clock synchronization in wireless networks~\cite{giridhar2006distributed}, signal reconstruction from phaseless measurements \cite{ABFM14,viswanathan2015fast}, and ranking of items based on noisy pairwise comparisons \cite{C16}.  For further discussions on the applications of phase synchronization, we refer the reader to~\cite{BBS16} and the references therein.

Although Problem~\eqref{eq:AS-ls} may seem to involve an objective function that is quartic in the decision variable $z$, it can actually be reformulated as a complex quadratic optimization problem with unit-modulus constraints.  Indeed, by writing the measurements~\eqref{eq:measure} more compactly as $C=(z^\star)(z^\star)^H+\Delta$,
where $(\cdot)^H$ denotes the Hermitian transpose and $\Delta$ is a Hermitian matrix whose diagonal entries are zero and the above-diagonal entries are given by $\{\Delta_{j\ell}:1\le j<\ell\le n\}$, and by noting that $|z_j\bar{z}_\ell|^2=1$ for $1\le j<\ell\le n$ because $z\in\T^n$, we see that Problem~\eqref{eq:AS-ls} is equivalent to
\begin{equation} \label{opt:QP} \tag{QP}
\hat{z} \in \arg\max_{z\in\T^n} \left\{ f(z)=z^HCz \right\}.
\end{equation}

As it turns out, Problem~\eqref{opt:QP} is NP-hard in general~\cite{TO98}.  Over the past two decades or so, many different approaches to tackling Problem~\eqref{opt:QP} have been proposed.  One popular approach is to apply the \emph{semidefinite relaxation} (SDR) technique, which will lead to a polynomial-time algorithm for computing a feasible but typically sub-optimal solution to Problem~\eqref{opt:QP} (see~\cite{LMS+10} for an overview of the technique).  Interestingly, the \emph{approximation accuracy} of the SDR solution, measured by the relative gap between the objective value of the SDR solution and the optimal value of Problem~\eqref{opt:QP}, can be established under various assumptions on $C$~\cite{SZY07,S10a}.  However, since our goal is to estimate the unknown phase vector $z^\star$, a more relevant measure of the quality of the SDR solution is its \emph{estimation error}, which intuitively can be defined as the distance between the SDR solution and the target phase vector $z^\star$.  Unfortunately, the aforementioned approximation accuracy results do not automatically translate into estimation error results. In an attempt to fill this gap, Bandeira et al.~\cite{BBS16} considered a Gaussian noise model and studied the estimation error of the SDR solution.  Specifically, suppose that the measurement noise takes the form $\Delta=\sigma W$, where $W$ is a Wigner matrix (i.e., a Hermitian random matrix whose diagonal entries are zero and the above-diagonal entries are i.i.d.~standard complex normal random variables) and $\sigma^2>0$ is the noise power.  It is shown in~\cite{BBS16} that if $\sigma=O(n^{1/4})$, then with high probability the standard SDR of Problem~\eqref{opt:QP} has a unique optimal solution that is of rank one; i.e., the SDR is tight.  This implies that a global maximizer $\hat{z}$ of Problem~\eqref{opt:QP}, which in this case is also a maximum likelihood estimator (MLE) of the target phase vector of $z^\star$, can be found in polynomial time.  Moreover, the expected squared $\ell_2$-estimation error of $\hat{z}$ is bounded above by $O(\sigma^2)$.  This matches (up to constants) the Cram\'er-Rao lower bound developed in~\cite{boumal2014cramer}, which applies to any unbiased estimator of $z^\star$.  As an aside, although large instances of the standard SDR of Problem~\eqref{opt:QP} may be costly to solve using interior-point methods, they can be solved quite efficiently in practice by numerical methods that exploit structure; see, e.g.,~\cite{WGY10,WMS11,WGS12}.  However, unlike interior-point methods, which are known to converge in polynomial time, most of these methods do not have convergence rate guarantees.

Besides the aforementioned SDR-based method, one can also employ the \emph{generalized power method} (GPM)~\cite{journee2010generalized} (see also~\cite{luss2013conditional}) to tackle Problem~\eqref{opt:QP}. When specialized to Problem~\eqref{opt:QP}, the GPM can be viewed as a gradient method on the manifold $\T^n$ and is much easier to implement than the SDR-based method. In a very recent work, Boumal~\cite{boumal2016nonconvex} analyzed the convergence behavior of the GPM under the same Gaussian noise model used in~\cite{BBS16} and showed that if $\sigma=O(n^{1/6})$, then with high probability the GPM will converge to a global maximizer of Problem~\eqref{opt:QP} when initialized by the eigenvector method in~\cite{Sing11}.  This result is significant, since in general the GPM may not even converge to a single point, let alone to a global optimizer of the problem at hand.  However, it does not give the rate at which the GPM converges to the global maximizer.  Moreover, compared with the result obtained for the SDR approach in~\cite{BBS16}, we see that the above result holds only in the more restrictive noise regime of $\sigma=O(n^{1/6})$.  Although numerical experiments in~\cite{BBS16,boumal2016nonconvex} indicate that both the SDR-based method and the GPM can find a global maximizer of Problem~\eqref{opt:QP} even when $\sigma$ is on the order of $n^{1/2}/{\rm polylog}(n)$, proving this rigorously remains an elusive task. 

Motivated by the preceding discussion, our goal in this paper is to gain a deeper understanding of the GPM when it is applied to Problem~\eqref{opt:QP} under the same Gaussian noise model used in~\cite{BBS16,boumal2016nonconvex}.  The starting point of our investigation is the following curious facts: Using Proposition~\ref{lem:q_norm}, which first appears in an earlier version of this paper, Boumal~\cite{boumal2016nonconvex} showed that modulo constants, the expected squared $\ell_2$-estimation error of the initial iterate of the GPM, viz.~the one produced by the eigenvector method in~\cite{Sing11}, already matches the Cram\'er-Rao bound.  Moreover, in the noise regime $\sigma=O(n^{1/6})$, we know by the results in~\cite{BBS16,boumal2016nonconvex} that the same is true for the limit point of the sequence of iterates generated by the GPM, as it is a global maximizer of Problem~\eqref{opt:QP}.  In view of these facts, it is natural to ask whether the intermediate iterates generated by the GPM also achieve an estimation error that is on the same order as the Cram\'er-Rao bound, and if so, whether the GPM actually reduces the estimation error in each iteration.  Our first contribution is to resolve both of these questions in the affirmative and to bound the rate at which the estimation error decreases in each iteration.  Specifically, we show that even at the noise level $\sigma=O(n^{1/2})$, the expected squared $\ell_2$-estimation errors of the iterates do not exceed $(c_1+c_2\tau^k)\sigma^2$, where $c_1,c_2>0$, $\tau\in(0,1)$ are some explicitly given constants and $k$ is the iteration counter; see the discussion after Corollary~\ref{cor:1}.  An interesting aspect of this result is that it holds regardless of whether the iterates converge or not (recall that the convergence result in~\cite{boumal2016nonconvex} holds only for noise level up to $O(n^{1/6})$).  Thus, from a statistical estimation viewpoint, one can terminate the GPM at any iteration and still obtain an estimator whose estimation error is on the same order as the Cram\'er-Rao bound.  Moreover, the leading constant in the estimation error becomes smaller as one runs more iterations of the GPM.  This explains in part the numerical observation in~\cite{boumal2016nonconvex} that the GPM can often return a good estimate of $z^\star$ even when the noise level is close to $O(n^{1/2})$. To the best of our knowledge, the bound we obtained on the $\ell_2$-estimation error of any accumulation point generated by the GPM holds under the least restrictive noise level requirement and is the best known to date in the Gaussian noise setting.

Next, we study the convergence behavior of the GPM when it is applied to Problem~\eqref{opt:QP}.  Our second contribution is to show that in the Gaussian noise setting, if $\sigma=O(n^{1/4})$ and the GPM is initialized by the eigenvector method, then with high probability the sequence of iterates generated by the GPM will converge linearly to a global maximizer of Problem~\eqref{opt:QP} (which is an MLE of $z^\star$); see Corollary~\ref{cor:global-conv}.  The significance of this result is twofold.  First, compared with the result in~\cite{boumal2016nonconvex}, the noise level requirement for the convergence of the GPM is relaxed from $O(n^{1/6})$ to $O(n^{1/4})$, thus matching the noise level requirement for the tightness of the SDR-based method.  Second, our result answers a question raised in~\cite{boumal2016nonconvex} concerning the convergence rate of the GPM and contributes to the growing literature on the design and analysis of fast algorithms for structured non-convex optimization problems (see, e.g.,~\cite{SQW15} and the references therein for an overview).  Key to our analysis is a new \emph{error bound} for Problem~\eqref{opt:QP}, which provides a computable estimate of the distance between any given point on $\T^n$ and the set of second-order critical points (which includes the global maximizers) of Problem~\eqref{opt:QP}; see Propositions~\ref{prop:err-bd} and~\ref{prop:EB-2crit}.  As a by-product, we show that every second-order critical point of Problem~\eqref{opt:QP} is still a global maximizer under a slightly less restrictive noise level requirement than~\cite{boumal2016nonconvex}; see the discussion after the proof of Proposition~\ref{prop:EB-2crit}.  We remark that error bounds have long played an important role in the convergence rate analysis of iterative methods; see, e.g.,~\cite{HZSL13,SZ15,ZS15,ZZS15,LP16,LWS16,Z16} for some recent developments.  However, most of the error bounds in the cited works are for convex optimization problems.  By contrast, our error bound is developed for the non-convex problem~\eqref{opt:QP}, which could be of independent interest.


We end this section by introducing the notations needed.
Let $\mathbf{1}$ denote the vector of all ones and $\mathbb{H}^n$ denote the set of $n\times n$ Hermitian matrices. For a complex vector $v\in\mathbb{C}^n$, let $\text{Diag}(v)$ denote the diagonal matrix whose diagonal elements are given by the entries of $v$, $|v|$ denote the vector of entry-wise moduli of $v$, and $\tfrac{v}{|v|}$ denote the vector of entry-wise normalizations of $v$; i.e.,
\begin{equation*}
({\rm Diag}(v))_{jj} = v_j, \quad |v|_j=|v_j|,\quad \left(\frac{v}{|v|}\right)_j=\left\{
\begin{array}{c@{\quad}l}
\tfrac{v_j}{|v_j|} &\text{if } v_j\neq 0,\\
\noalign{\smallskip}
0 & \text{otherwise}.
\end{array}
\right.
\end{equation*}
For a complex matrix $M\in\mathbb{C}^{n\times n}$, let $\text{diag}(M)$ denote the vector whose entries are the diagonal elements of $M$, $\|M\|_{\text{op}}$ denote its operator norm, and $\|M\|_F$ denote its Frobenius norm.

Since the measurements $\{C_{j\ell}:1\le j<\ell\le n\}$ in~\eqref{eq:measure} are invariant under multiplication of a common phase to the target phase vector $z^\star$, we can only identify $z^\star$ up to a global phase.  This motivates us to define the $\ell_q$-distance (where $q\in[1,\infty]$) between two phase vectors $w,z\in\mathbb{T}^n$ by
$$d_q(w,z)=\min_{\theta\in[0,2\pi)}\|w-e^{i\theta}z\|_q. $$

\section{Preliminaries}

In this section, we review the GPM for solving Problem~\eqref{opt:QP} and collect some basic facts that will be used in our subsequent analysis.

The GPM is an iterative method that was introduced in~\cite{journee2010generalized} for maximizing a convex function over a compact set.  In each iteration of the GPM, an affine minorant of the objective function at the current iterate is maximized over the feasible set to obtain the next iterate.  When specialized to Problem~\eqref{opt:QP}, the maximization performed in each iteration admits a closed-form solution, and the GPM takes the following form:

\begin{algorithm}[H]
\caption{Generalized Power Method for Problem~\eqref{opt:QP}} \label{alg:GP}
\begin{algorithmic}[1]
\STATE input: objective matrix $C\in\mathbb{H}^n$, step size $\alpha>0$, initial point $z^0\in\T^n$
\FOR {$k=0,1,\dots$}
\IF{termination criterion is met}
\STATE return $z^k$
\ELSE
\STATE $w^k \gets \left(I+\tfrac{\alpha}{n} C\right)z^k$
\STATE $z^{k+1} \gets \tfrac{w^k}{|w^k|}$
\ENDIF
\ENDFOR
\end{algorithmic}
\end{algorithm}

Algorithm~\ref{alg:GP} can be viewed as a projected gradient method (see lines 6-7), though it is not necessarily a Riemannian gradient method on the manifold $\T^n$; see~\cite[Remark 1]{boumal2016nonconvex}.  Due to the non-convexity of Problem~\eqref{opt:QP}, given an arbitrary initial point, Algorithm~\ref{alg:GP} may not converge to any useful point (if it converges at all).  To tackle this issue, Boumal~\cite{boumal2016nonconvex} proposed to use the \emph{eigenvector estimator} $v_C\in\T^n$ (cf.~\cite{Sing11}) to initialize Algorithm~\ref{alg:GP}.  Specifically, let $u\in\C^n$ be a leading eigenvector of $C\in\mathbb{H}^n$ and $a\in\C^n$ be any vector satisfying $a^Hu\not=0$.  Then, the vector $v_C$ is defined by
\begin{equation}\label{eq:init}
(v_C)_j = \left\{
\begin{array}{c@{\quad}l}
\tfrac{u_j}{|u_j|} & \text{if}\ u_j\neq 0,\\
\noalign{\smallskip}
\tfrac{a^Hu}{|a^Hu|} & \text{otherwise}
\end{array}
\right. \quad\mbox{for } j=1,\ldots,n.
\end{equation}
As shown in~\cite{boumal2016nonconvex}, the advantage of initializing Algorithm~\ref{alg:GP} with $z^0=v_C$ is twofold.  First, the vector $v_C$ is close to the target phase vector $z^\star$ in the following sense:
\begin{fact} (\cite[Lemma 6]{boumal2016nonconvex}) \label{lem:init}
Let $v_C\in\T^n$ be given by \eqref{eq:init}. Then, we have
\begin{equation*}
d_2(v_C,z^\star) \leq \frac{8\|\Delta\|_{\rm op}}{\sqrt{n}}.
\end{equation*}
\end{fact}
\noindent Second, under some mild assumptions on the measurement noise $\Delta$ and step size $\alpha$, the iterates generated by Algorithm~\ref{alg:GP} will converge to a global maximizer of Problem~\eqref{opt:QP}:
\begin{fact} \label{fact:global-convergence} (\cite[Theorem 3]{boumal2016nonconvex})
Suppose that (i) the measurement noise $\Delta$ satisfies $\|\Delta\|_{\rm op}=O(n^{2/3})$ and $\|\Delta z^\star\|_\infty = O(n^{2/3}\sqrt{\log n})$, (ii) the step size $\alpha$ satisfies $\alpha \le \tfrac{n}{\|\Delta\|_{\rm op}}$, and (iii) the initial point $z^0$ is given by $z^0=v_C$.  Then, the iterates generated by Algorithm~\ref{alg:GP} will converge to a global maximizer of Problem~\eqref{opt:QP}.
\end{fact}

It should be noted that even allowing for the multiplication of a common phase, a global maximizer $\hat{z}$ of Problem~\eqref{opt:QP} may not equal to the target phase vector $z^\star$.  Thus, an immediate question is whether global maximizers of Problem~\eqref{opt:QP} are close to $z^\star$.  The following result shows that the answer is affirmative:
\begin{fact} (\cite[Lemma 4.1]{BBS16}) \label{lem:1}
Suppose that $\hat{z}\in\C^n$ satisfies $\|\hat{z}\|_2^2=n$ and $f(z^\star) \leq f(\hat{z})$ (e.g., if $\hat{z}$ is a global maximizer of Problem~\eqref{opt:QP}).  Then, we have
$$
d_2(\hat{z},z^\star) = \sqrt{2\left( n-|\hat{z}^Hz^\star| \right)}\leq \frac{4 \|\Delta\|_{\rm op}}{\sqrt{n}}.
$$
\end{fact}

Lastly, let us record a useful property of Algorithm~\ref{alg:GP}.  Recall that $\tilde{z}\in\T^n$ is a \emph{second-order critical point} of Problem~\eqref{opt:QP} if $w^HS(\tilde{z})w \ge 0$ for all $w\in T_{\tilde{z}}\T^n$, where
$$ S(z) = \Re\left\{ \mbox{Diag}\left( \mbox{diag}(Czz^H) \right) \right\} - C $$
and
$$ T_z\T^n = \left\{ w \in \C^n : \Re\left\{ w_i\bar{z}_i \right\} = 0 \mbox{ for } i=1,\ldots,n \right\} $$
is the tangent space to $\T^n$ at $z\in\T^n$; see~\cite{BBS16,boumal2016nonconvex}.  By considering the second-order necessary optimality conditions of Problem~\eqref{opt:QP}, it can be shown that every global maximizer of Problem~\eqref{opt:QP} is a second-order critical point.  The following result asserts that a second-order critical point of Problem~\eqref{opt:QP} is (i) a fixed point of Algorithm~\ref{alg:GP} and (ii) close to the target phase vector $z^\star$ if the measurement noise $\Delta$ is not too large.

\ignore{As shown in \cite{boumal2016nonconvex}, the accumulation points of Algorithm \ref{alg:GP} are divided into two categories: the good fixed points and the bad fixed points. Roughly speaking, the good fixed points are those close to the true signal $z$, whereas the bad ones are those bounded away from $z$. To avoid GPM from converging to those unfavorable local points, a proper initialization is in order.}

\begin{fact} (\cite[Lemmas 7, 14, 15, and 16]{boumal2016nonconvex}) \label{lem:fix-pt}
Let $\tilde{z}\in\T^n$ be any second-order critical point of Problem~\eqref{opt:QP} and $\tilde{C}= \tfrac{n}{\alpha}\left( I+\tfrac{\alpha}{n}C \right) = C+\tfrac{n}{\alpha}I$. Then, for any $\alpha>0$,
\begin{equation*}
|(C\tilde{z})_j| = (C\tilde{z})_j \overline{(\tilde{z}_j)} \quad\mbox{and}\quad | (\tilde{C}\tilde{z})_j | = (\tilde{C}\tilde{z})_j \overline{(\tilde{z}_j)} \quad\mbox{for } j=1,\ldots,n.
\end{equation*}
Consequently, we have $\tilde{z}^HC\tilde{z}=\|C\tilde{z}\|_1$, $\tilde{z}^H\tilde{C}\tilde{z}=\|\tilde{C}\tilde{z}\|_1$, and
$$\left({\rm Diag}(|\tilde{C}\tilde{z}|)-\tilde{C}\right)\tilde{z} = \left({\rm Diag}(|C\tilde{z}|)-C\right)\tilde{z} = \bz.$$
Moreover, if $\|\Delta\|_{\rm op} \le \tfrac{n}{16}$, then
$$ |(z^\star)^H\tilde{z}| \ge n-4\|\Delta\|_{\rm op} \quad\mbox{and}\quad d_2(\tilde{z},z^\star) \le \sqrt{8\|\Delta\|_{\rm op}}. $$
\end{fact}

To distinguish the different points of interest on $\T^n$, we shall reserve the notations $z^\star$, $\hat{z}$, and $\tilde{z}$ to denote the target phase vector, a global maximizer of Problem~\eqref{opt:QP}, and a second-order critical point of Problem~\eqref{opt:QP}, respectively in the sequel.

\section{Estimation Performance of the GPM}
Facts~\ref{lem:init} and~\ref{lem:1} show that both the eigenvector estimator $v_C$ and global maximizers of Problem~\eqref{opt:QP} are close to the target phase vector $z^\star$.  In this section, we show that the same is true for all intermediate iterates of Algorithm~\ref{alg:GP}.
In fact, we establish a stronger result: We show that  the $\ell_2$- and $\ell_\infty$-estimation errors of the iterates decrease in each iteration of Algorithm~\ref{alg:GP} and provide explicit bounds on the rates of decrease.

To begin, let us introduce our first result, which concerns the $\ell_2$-estimation errors of the iterates:
\begin{theorem}\label{thm:2-norm}
Suppose that (i) the measurement noise $\Delta$ satisfies $\|\Delta\|_{\rm op}\le\tfrac{n}{16}$, (ii) the step size $\alpha$ satisfies $\alpha\ge2$, and (iii) the initial point $z^0$ is given by $z^0=v_C$.  Then, the sequence of iterates $\{z^k\}_{k\ge0}$ generated by Algorithm~\ref{alg:GP} satisfies
$$ d_2(z^{k+1},z^\star) \le \mu^{k+1}\cdot d_2(z^0,z^\star) +  \frac{\nu}{1-\mu} \frac{8\|\Delta\|_{\rm op}}{\sqrt{n}} \le \left( \mu^{k+1} + \frac{\nu}{1-\mu} \right) \frac{8\|\Delta\|_{\rm op}}{\sqrt{n}} $$
for $k=0,1,\ldots$, where
\begin{equation} \label{eq:mu-nu}
\mu = \frac{16(\alpha\|\Delta\|_{\rm op}+n)}{(7\alpha+8)n} < 1, \quad \nu = \frac{2\alpha}{7\alpha+8}.
\end{equation}
\end{theorem}

Theorem~\ref{thm:2-norm} has two noteworthy features.  First, it does not assume that Algorithm~\ref{alg:GP} converges. Second, it provides a bound on the $\ell_2$-estimation error of each iterate generated by Algorithm~\ref{alg:GP}.  As such, one can terminate Algorithm~\ref{alg:GP} at any iteration and still has a guarantee on the quality of the estimator.

To further illustrate the usefulness of Theorem \ref{thm:2-norm}, recall from Facts~\ref{lem:init} and~\ref{lem:1} that the $\ell_2$-estimation errors of the initial point $v_C$ and the global maximizers of Problem~\eqref{opt:QP} are bounded above by $\tfrac{8\|\Delta\|_{\rm op}}{\sqrt{n}}$ and $\tfrac{4\|\Delta\|_{\rm op}}{\sqrt{n}}$, respectively.  Now, if we take $\alpha=4$ in Algorithm~\ref{alg:GP}, then under the assumptions of Theorem~\ref{thm:2-norm}, we have $\mu \le \tfrac{5}{9}$ and $\nu = \tfrac{2}{9}$. This implies that any accumulation point $z^\infty$ generated by Algorithm~\ref{alg:GP} satisfies 
$$ d_2(z^{\infty},z^\star)\leq \frac{4\|\Delta\|_{\rm op}}{\sqrt{n}},$$ 
which matches the bound on the $\ell_2$-estimation error of any global maximizer of Problem~\eqref{opt:QP}.  Furthermore, if we let $\alpha\rightarrow\infty$, which can be interpreted as using the update $z^{k+1} \gets \tfrac{Cz^k}{|Cz^k|}$ in line 7 of Algorithm~\ref{alg:GP}, then 
$$ d_2(z^{k+1},z^\star) \le \left( \left(\frac{1}{7}\right)^{k+1} + \frac{1}{3} \right) \frac{8\|\Delta\|_{\rm op}}{\sqrt{n}} $$
for $k=0,1,\ldots$.  In this case, our bound is even better than that in Fact~\ref{lem:1} when $k$ is sufficiently large.

Next, we present our result on the $\ell_\infty$-estimation errors of the iterates:
\begin{theorem}\label{thm:infty-norm}
Under the same assumptions as Theorem~\ref{thm:2-norm}, the sequence of iterates $\{z^k\}_{k\ge0}$ generated by Algorithm~\ref{alg:GP} satisfies
$$ d_\infty(z^{k+1},z^\star) \le \gamma^{k+1}\cdot d_\infty(z^0,z^\star) + \frac{\zeta\cdot\mu^k}{1-\gamma/\mu} + \frac{\omega}{1-\gamma} $$
for $k=0,1,2,\dots$, where
$$ \gamma = \frac{16}{7\alpha+8} < 1, \,\, \zeta = \frac{128\alpha\|\Delta\|_{\rm op}^2}{(7\alpha+8)n^{3/2}}, \,\, \omega = \frac{16\alpha}{7\alpha+8}\left( \frac{\nu}{1-\mu}\frac{8\|\Delta\|_{\rm op}^2}{n^{3/2}} + \frac{\|\Delta z^\star\|_\infty}{n} \right), $$
and $\mu,\nu$ are given in~\eqref{eq:mu-nu}, so that $\gamma/\mu<1$.
\end{theorem}

To prove Theorems~\ref{thm:2-norm} and~\ref{thm:infty-norm}, we need the following technical results:
\begin{proposition} \label{lem:q_norm}
For any $w\in \mathbb{C}^n$, $z\in \mathbb{T}^n$, and $q\in[1,\infty]$, we have
$$
\left\| \frac{w}{|w|}-z \right\|_{q} \leq 2 \|w-z\|_{q}.
$$
\end{proposition}
\begin{proof}
Without loss of generality, we may assume that $z=\mathbf{1}$.  By definition of $\tfrac{w}{|w|}$, it suffices to show that for $j=1,\dots,n$,
\begin{equation*}
\left|\left(\frac{w}{|w|}\right)_j-1 \right|\leq 2|w_j-1|.
\end{equation*}
The above inequality holds trivially if $w_j=0$.  Hence, we may focus on the case where $w_j\neq 0$. We claim that
\begin{equation*}
|e^{i\phi}-1|\leq 2|re^{i\phi}-1|\quad \mbox{for any } \phi\in [0,2\pi) \mbox{ and } r\geq0.
\end{equation*}
To prove this, observe that $|re^{i\phi}-1|^2=r^2-2r\cos\phi+1$.  Thus, we have
\begin{equation*}
\arg\min_{r\geq 0}|re^{i\phi}-1|^2=
\left\{
\begin{array}{c@{\quad}l}
0 & \mbox{if } \phi\in [\tfrac{\pi}{2},\tfrac{3\pi}{2}],\\
\noalign{\smallskip}
\cos\phi & \mbox{if } \phi\in[0,\tfrac{\pi}{2})\cup(\tfrac{3\pi}{2},2\pi),
\end{array}
\right.
\end{equation*}
from which it follows that
\begin{equation} \label{eq:min-val}
\min_{r\geq 0}|re^{i\phi}-1|^2=
\left\{
\begin{array}{c@{\quad}l}
1 & \mbox{if } \phi\in [\tfrac{\pi}{2},\tfrac{3\pi}{2}],\\
\noalign{\smallskip}
\sin^2\phi & \mbox{if } \phi\in[0,\tfrac{\pi}{2})\cup(\tfrac{3\pi}{2},2\pi).
\end{array}
\right.
\end{equation}
Now, for $\phi\in [\tfrac{\pi}{2},\tfrac{3\pi}{2}]$, by the triangle inequality and~\eqref{eq:min-val}, we have
$$|e^{i\phi}-1|\leq 2\leq 2|re^{i\phi}-1|\quad \text{for any } r\geq0. $$
On the other hand, for $\phi\in[0,\tfrac{\pi}{2})\cup(\tfrac{3\pi}{2},2\pi)$, we use the half-angle formula and~\eqref{eq:min-val} to get
$$|e^{i\phi}-1|=\sqrt{2(1-\cos\phi)}=2\left|\sin\frac{\phi}{2}\right|\leq 2\left| \sin\phi\right|\leq 2|re^{i\phi}-1|\quad \text{for any } r\geq0.$$
Combining the above two cases, the proof is completed.
\end{proof}

\begin{proposition} \label{prop:iter-bd}
Let $\{z^k\}_{k\ge0}$ be the sequence of iterates generated by Algorithm~\ref{alg:GP} with $\alpha>0$. For $q\in[1,\infty]$ and $k=0,1,\ldots$, define
\begin{eqnarray*}
\theta_k &=& \arg\min_{\theta\in[0,2\pi)} \|z^k-e^{i\theta}z^\star\|_q, \\
\noalign{\smallskip}
\epsilon^k &=& e^{-i\theta_k}(z^k-e^{i\theta_k}z^\star), \\
\noalign{\smallskip}
\beta_k &=& 1+\alpha+\frac{\alpha}{n} (z^\star)^H(\epsilon^k).
\end{eqnarray*}
Then, for any $r\in\C$ and $k\in\{0,1,\ldots\}$, we have 
$$ d_q(z^{k+1},z^\star) \le 2\|rg^k-z^\star\|_q, $$
where
$$ g^k = \beta_kz^\star + \left( I+\frac{\alpha}{n}\Delta \right)\epsilon^k + \frac{\alpha}{n}\Delta z^\star. $$
\end{proposition}
\begin{proof}
Consider a fixed $k\in\{0,1,\ldots\}$.  By definition, we have
\begin{eqnarray*}
w^k &=& \left( I+\frac{\alpha}{n} C \right) z^k \\
\noalign{\smallskip}
&=& e^{i\theta_k} \left( I+\frac{\alpha}{n}((z^\star)(z^\star)^H+\Delta) \right)(z^\star+\epsilon^k) \\
\noalign{\smallskip}
&=&\left[\left( 1+\alpha+\frac{\alpha}{n}(z^\star)^H(\epsilon^k) \right)z^\star + \left( I+\frac{\alpha}{n}\Delta \right)\epsilon^k + \frac{\alpha}{n}\Delta z^\star\right]e^{i\theta_k} \\
\noalign{\smallskip}
&=& g^k e^{i\theta_k}.
\end{eqnarray*}
Since $z^{k+1}=\tfrac{w^k}{|w^k|}$, it follows from Proposition~\ref{lem:q_norm} that for any $r\in\C\setminus\{0\}$,
$$ d_q(z^{k+1},z^\star) \le \left\| \frac{g^k}{|g^k|}-z^\star \right\|_q = \left\| \frac{rg^k}{|rg^k|}-z^\star \right\|_q \le 2\|rg^k-z^\star\|_q. $$
Since the above inequality holds for all $r\in\C\setminus\{0\}$, by taking $r\rightarrow0$, we see that it holds for $r=0$ as well.
\end{proof}

We are now ready to prove Theorems~\ref{thm:2-norm} and~\ref{thm:infty-norm}.
\begin{proof}[Proof of Theorem~\ref{thm:2-norm}] We prove by induction that for $k=0,1,\ldots$, the following inequalities hold:
\begin{eqnarray}
\|\epsilon^k\|_2 &\le& \frac{\sqrt{n}}{2}, \label{eq:eps-bd} \\
\noalign{\smallskip}
d_2(z^{k+1},z^\star) &\le& \mu \cdot d_2(z^k,z^\star) + \nu\cdot\frac{8\|\Delta\|_{\rm op}}{\sqrt{n}}. \label{eq:recur}
\end{eqnarray}
Indeed, by the definition of $\epsilon^0$, Fact~\ref{lem:init}, and the assumption that $\|\Delta\|_{\rm op}\le\tfrac{n}{16}$, we have $\|\epsilon^0\|_2 = d_2(z^0,z^\star) \le \tfrac{\sqrt{n}}{2}$.  This implies that
\begin{eqnarray}
|\beta_0| &\ge& \left| 1+\alpha+\frac{\alpha}{n}\Re\left( (z^\star)^H(\epsilon^0) \right) \right| \nonumber \\
\noalign{\smallskip}
&=& \left|  1+\alpha+\frac{\alpha}{2n} \left( \|z^\star+\epsilon^0\|_2^2 - \|z^\star\|_2^2 - \|\epsilon^0\|_2^2 \right) \right| \nonumber \\
\noalign{\smallskip}
&=& \left|  1+\alpha+\frac{\alpha}{2n} \left( \|z^0\|_2^2 - \|z^\star\|_2^2 - \|\epsilon^0\|_2^2 \right) \right| \nonumber \\
\noalign{\smallskip}
&\ge& 1+\frac{7\alpha}{8}, \label{eq:beta-bd}
\end{eqnarray}
where the last inequality follows from the fact that $\|z^0\|_2^2=\|z^\star\|_2^2=n$ and $\|\epsilon^0\|_2\le\tfrac{\sqrt{n}}{2}$.  Hence, by taking $r=\beta_0^{-1}$ (which is well-defined) and $q=2$ in Proposition~\ref{prop:iter-bd} and using~\eqref{eq:beta-bd}, we have
\begin{eqnarray}
d_2(z^1,z^\star) &\le& 2\left\| \beta_0^{-1} \left( I+\frac{\alpha}{n}\Delta \right)\epsilon^0 + \beta_0^{-1}\frac{\alpha}{n}\Delta z^\star \right\|_{2} \nonumber \\
\noalign{\smallskip}
&\le& 2|\beta_0^{-1}| \cdot \left[ \left\| \left(I+\frac{\alpha}{n}\Delta\right)\epsilon^0 \right\|_{2} + \frac{\alpha}{n} \| \Delta z^\star\|_{2} \right] \nonumber \\
\noalign{\smallskip}
&\le& \frac{16}{7\alpha+8} \left[ \left(1+\frac{\alpha}{n}\|\Delta\|_{\rm op} \right)\|\epsilon^0\|_{2}+\frac{\alpha}{\sqrt{n}}\|\Delta\|_{\rm op} \right] \nonumber \\
\noalign{\smallskip}
&=& \mu\cdot d_2(z^0,z^\star) + \nu \cdot \frac{8\|\Delta\|_{\rm op}}{\sqrt{n}}. \label{eq:recur-pf}
\end{eqnarray}
Now, suppose that~\eqref{eq:eps-bd} and~\eqref{eq:recur} hold for some $k\ge0$.  By the inductive hypothesis and the assumption that $\|\Delta\|_{\rm op}\le\tfrac{n}{16}$ and $\alpha\ge2$, we have
\begin{eqnarray*}
\|\epsilon^{k+1}\|_2 &=& d_2(z^{k+1},z^\star) \,\,\,\le\,\,\, \mu\cdot d_2(z^k,z^\star)+\nu \cdot \frac{8\|\Delta\|_{\rm op}}{\sqrt{n}} \\
\noalign{\smallskip}
&=& \mu\|\epsilon^k\|_2 + \nu \cdot \frac{8\|\Delta\|_{\rm op}}{\sqrt{n}} \\
\noalign{\smallskip}
&\le& \frac{8(\alpha\|\Delta\|_{\rm op}+n)}{(7\alpha+8)\sqrt{n}} + \frac{16\alpha\|\Delta\|_{\rm op}}{(7\alpha+8)\sqrt{n}} \\
\noalign{\smallskip}
&\le& \frac{\sqrt{n}}{2}.
\end{eqnarray*}
Using the same argument as the derivation of the inequality~\eqref{eq:beta-bd}, we have $|\beta_{k+1}|\ge1+\tfrac{7\alpha}{8}$.  Hence, following the same derivation as the inequality~\eqref{eq:recur-pf}, we obtain $d_2(z^{k+2},z^\star) \le \mu\cdot d_2(z^{k+1},z^\star) + \nu\cdot \tfrac{8\|\Delta\|_{\rm op}}{\sqrt{n}}$.  This completes the inductive step.

To complete the proof of Theorem~\ref{thm:2-norm}, it remains to unroll~\eqref{eq:recur} and use Fact~\ref{lem:init}.
\end{proof}

\begin{proof}[Proof of Theorem \ref{thm:infty-norm}]
From the proof of Theorem \ref{thm:2-norm}, we have $\|\epsilon^k\|_2\leq \tfrac{\sqrt{n}}{2}$ and $|\beta_k^{-1}|\leq\tfrac{8}{7\alpha+8}$ for $k=0,1,\ldots$.  By taking $r=\beta_k^{-1}$ and $q=\infty$ in Proposition~\ref{prop:iter-bd} and using Theorem~\ref{thm:2-norm}, we compute
\begin{eqnarray}
& & d_\infty(z^{k+1},z^\star) \nonumber \\
\noalign{\smallskip}
&\le& 2\|\beta_k^{-1}g^k-z^\star\|_\infty \nonumber \\
\noalign{\smallskip}
&\le& 2|\beta_k^{-1}| \cdot \left[ \|\epsilon^k\|_\infty + \frac{\alpha}{n} \left( \|\Delta\epsilon^k\|_\infty + \|\Delta z^\star\|_\infty \right) \right] \nonumber \\
\noalign{\smallskip}
&\le& \frac{16}{7\alpha+8}\left[ d_\infty(z^k,z^\star) + \frac{\alpha}{n} \left( \|\Delta\epsilon^k\|_2 + \|\Delta z^\star\|_\infty \right) \right] \nonumber \\
\noalign{\smallskip}
&\le& \frac{16}{7\alpha+8}\left[ d_\infty(z^k,z^\star) + \frac{\alpha}{n} \left( \|\Delta\|_{\rm op} \cdot \|\epsilon^k\|_2 + \|\Delta z^\star\|_\infty \right) \right] \nonumber \\
\noalign{\smallskip}
&\le& \frac{16}{7\alpha+8}\left[ d_\infty(z^k,z^\star) + \frac{\alpha}{n} \left( \left( \mu^k+\frac{\nu}{1-\mu} \right) \frac{8\|\Delta\|_{\rm op}^2}{\sqrt{n}} + \|\Delta z^\star\|_\infty \right) \right] \nonumber \\
\noalign{\smallskip}
&=& \gamma \cdot d_\infty(z^k,z^\star) + \zeta\cdot\mu^k + \omega. \label{ineq:infty_norm}
\end{eqnarray}
Since $\alpha\geq2$, we have $\gamma\in(0,1)$ and 
$$ \frac{\gamma}{\mu}=\frac{16}{(7\alpha+8)}\frac{(7\alpha+8)n}{16(\alpha\|\Delta\|_{\rm op}+n)}=\frac{n}{\alpha\|\Delta\|_{\rm op}+n}\in (0,1). $$
It follows from~\eqref{ineq:infty_norm} that
\begin{eqnarray*}
d_\infty(z^{k+1},z^\star) &\le& \gamma \cdot d_\infty(z^k,z^\star) + \zeta\cdot\mu^k + \omega \\
\noalign{\smallskip}
&\le& \gamma^{k+1}\cdot d_\infty(z^0,z^\star) + \zeta\sum_{j=0}^k \gamma^j\mu^{k-j} + \omega\sum_{j=0}^k\gamma^j \\
\noalign{\smallskip}
&\le& \gamma^{k+1}\cdot d_\infty(z^0,z^\star) + \frac{\zeta\cdot\mu^k}{1-\gamma/\mu} + \frac{\omega}{1-\gamma}.
\end{eqnarray*}
This completes the proof.
\end{proof}

By specializing the above results to the Gaussian noise setting, we obtain the following corollary:
\begin{corollary}\label{cor:1}
Suppose that the measurement noise $\Delta$ takes the form $\Delta=\sigma W$, where $\sigma^2>0$ is the noise power satisfying $\sigma\in\left(0,\tfrac{\sqrt{n}}{48}\right]$ and $W\in\mathbb{H}^n$ is a Wigner matrix.  Suppose further that the step size $\alpha$ satisfies $\alpha\ge2$ and the initial point $z^0$ is given by $z^0=v_C$.  Then, with probability at least $1-2n^{-5/4}-2e^{-n/2}$, the sequence of iterates $\{z^k\}_{k\ge0}$ generated by Algorithm~\ref{alg:GP} satisfies
\begin{eqnarray*}
d_2(z^{k+1},z^\star) &\le& \left(\frac{\alpha+16}{7\alpha+8}\right)^{k+1}d_2(z^0,z^\star) + \frac{24\alpha}{3\alpha-4}\sigma, \\
\noalign{\smallskip}
d_\infty(z^{k+1},z^\star) &\le& \left(\frac{16}{7\alpha+8}\right)^{k+1}d_\infty(z^0,z^\star) + \left( \frac{\alpha+16}{7\alpha+8}\right)^{k+1}\frac{\sqrt{n}}{2} \\
\noalign{\smallskip}
& & \,\,+\,\, \frac{48\alpha}{7\alpha+8}\left( \sqrt{\log n}+\frac{24\alpha}{3\alpha-4}\sigma \right)\frac{\sigma}{\sqrt{n}}
\end{eqnarray*}
for $k=0,1,\ldots$.
\end{corollary}

\begin{proof}
By~\cite[Proposition 3.3]{BBS16}, we have $\|W\|_{\rm op}\le3\sqrt{n}$ and $\|Wz^\star\|_\infty \le 3\sqrt{n\log n}$ with probability at least $1-2n^{-5/4}-2e^{-n/2}$.  The result then follows by combining these estimates with the bounds in Theorems~\ref{thm:2-norm} and~\ref{thm:infty-norm}.
\end{proof}

\noindent Note that by Fact~\ref{lem:init} and~\cite[Proposition 3.3]{BBS16}, we have $d_2(z^0,z^\star)\le24\sigma$ with high probability.  Hence, for $\alpha>4$ and $k$ sufficiently large, the bound on the $\ell_2$-estimation error will be strictly less than $12\sigma$, which is better than that obtained from Fact~\ref{lem:1} for any maximum likelihood estimator (which is a global maximizer of Problem~\eqref{opt:QP}) of the target phase vector $z^\star$.  Furthermore,  for $k=0,1,\ldots$, since $d_2(z^k,z^\star)^2 \le 2n$, we have
\begin{eqnarray*}
& & \mathbb{E}\left[d_2(z^{k+1},z^\star)^2\right] \\
\noalign{\smallskip}
&\le& 1152\left(1-2n^{-5/4}-2e^{-n/2}\right)\left[ \left(\frac{\alpha+16}{7\alpha+8}\right)^{2(k+1)}+ \left(\frac{\alpha}{3\alpha-4}\right)^2 \right]\sigma^2 \\
\noalign{\smallskip}
&& \,\,+\,\, \left(2n^{-5/4}+2e^{-n/2}\right)(2n)  \\
\noalign{\smallskip}
&\le& (c_1+c_2\tau^k)\sigma^2
\end{eqnarray*}
for some constants $c_1,c_2>0$ and $\tau\in(0,1)$.  This shows that the expected squared $\ell_2$-estimation errors of the iterates generated by Algorithm~\ref{alg:GP} are all on the order of $\sigma$, which matches the Cram\'{e}r-Rao bound developed in~\cite{boumal2014cramer}.  It is worth noting that the above conclusions hold even when the noise level is $\sigma=O(n^{1/2})$, which is the least restrictive among similar results in the literature; cf.~\cite{BBS16,boumal2016nonconvex}.  Our result explains in part the excellent numerical estimation performance of the GPM observed in~\cite{boumal2016nonconvex} even when the noise level is close to $O(n^{1/2})$.


\section{Convergence Rate of the GPM}
Although the results in the previous section show that Algorithm~\ref{alg:GP} generates increasingly accurate (in the $\ell_2$ and $\ell_\infty$ sense) estimators of the target phase vector $z^\star$, they do not shed any light on its convergence behavior.  On the other hand, recall from Fact~\ref{fact:global-convergence} that the sequence of iterates generated by Algorithm~\ref{alg:GP} will converge to a global maximizer of Problem~\eqref{opt:QP} under suitable assumptions on the measurement noise $\Delta$ and step size $\alpha$.  However, it does not give the rate of convergence.  In this section, we prove that under weaker assumptions than those of Fact~\ref{fact:global-convergence}, both the sequence of iterates and the associated sequence of objective values generated by Algorithm~\ref{alg:GP} will converge \emph{linearly} to a global maximizer and the optimal value of Problem~\eqref{opt:QP}, respectively.  Specifically, we have the following result:
\begin{theorem}\label{thm:convergence-rate}
  Suppose that (i) the measurement noise $\Delta$ satisfies $\|\Delta\|_{\rm op}\le\tfrac{n^{3/4}}{312}$ and $\|\Delta z^\star\|_\infty \le \tfrac{n}{24}$, (ii) the step size $\alpha$ satisfies $\alpha\in\left[ 4,\tfrac{n}{\|\Delta\|_{\rm op}} \right)$, and (iii) the initial point $z^0$ is given by $z^0=v_C$.  Then, the sequence of iterates $\{z^k\}_{k\ge0}$ generated by Algorithm~\ref{alg:GP} satisfies
  \begin{eqnarray*}
    f(\hat{z}) - f(z^k) &\le& \left( f(\hat{z})-f(z^0) \right) \lambda^k, \\
    \noalign{\smallskip}
    d_2(z^k,\hat{z}) &\leq& a \left( f(\hat{z})-f(z^0) \right)^{1/2} \lambda^{k/2}
  \end{eqnarray*}
  for $k=0,1,\ldots$, where $a>0,\lambda\in(0,1)$ are quantities that depend only on $n$ and $\alpha$, and $\hat{z}$ is any global maximizer of Problem~\eqref{opt:QP}.
\end{theorem}

Theorem~\ref{thm:convergence-rate} improves upon Fact~\ref{fact:global-convergence} in two aspects.  First, Theorem~\ref{thm:convergence-rate} holds under a less restrictive requirement on the measurement noise $\Delta$.  Specifically, it requires that $\|\Delta\|_{\rm op}=O(n^{3/4})$ and $\|\Delta z^\star\|_\infty=O(n)$, while Fact~\ref{fact:global-convergence} requires that $\|\Delta\|_{\rm op}=O(n^{2/3})$ and $\|\Delta z^\star\|_\infty=O(n^{2/3}\sqrt{\log n})$. Second, Theorem~\ref{thm:convergence-rate} is more quantitative than Fact~\ref{fact:global-convergence} in the sense that it also gives the rate at which Algorithm~\ref{alg:GP} converges.  Consequently, we resolve an open question raised in~\cite{boumal2016nonconvex}.

The proof of Theorem~\ref{thm:convergence-rate} consists of two main parts.  The first, which is the more challenging part, is to establish the following \emph{error bound} for Problem~\eqref{opt:QP}.  Such a bound provides a computable estimate of the distance between any point in a neighborhood of $z^\star$ and the set of global maximizers of Problem~\eqref{opt:QP}, which could be of independent interest.
\begin{proposition} \label{prop:err-bd}
Let $\Sigma:\T^n\rightarrow\mathbb{H}^n$ and $\rho:\T^n\rightarrow\R_+$ be defined as
$$
\Sigma(z)={\rm Diag}(|\tilde{C}z|)-\tilde{C}, \quad
\rho(z)= \|\Sigma(z)z\|_2 = \left\| \left({\rm Diag}(|\tilde{C}z|)-\tilde{C}\right)z \right\|_2,
$$
where $\tilde{C}=C+\frac{n}{\alpha}I$.  Under the assumptions of Theorem~\ref{thm:convergence-rate}, for any point $z\in\T^n$ satisfying $d_2(z,z^\star)\le\tfrac{\sqrt{n}}{2}$ and any global maximizer $\hat{z}\in\T^n$ of Problem~\eqref{opt:QP}, we have
$$ d_2(z,\hat{z}) \le \frac{8}{n}\rho(z). $$
\end{proposition}
Before we prove Proposition~\ref{prop:err-bd}, several remarks are in order.  First, recall from Fact~\ref{lem:1} that every global maximizer $\hat{z}$ of Problem~\eqref{opt:QP} satisfies $d_2(\hat{z},z^\star)\le\tfrac{\sqrt{n}}{2}$ whenever $\|\Delta\|_{\rm op}\le\tfrac{n}{8}$. Together with Proposition~\ref{prop:err-bd}, this shows that up to a global phase, Problem~\eqref{opt:QP} has a unique global maximizer.  Second, the proof of Theorem~\ref{thm:2-norm} reveals that the sequence of iterates $\{z^k\}_{k\ge0}$ generated by Algorithm~\ref{alg:GP} satisfies $d_2(z^k,z^\star)\le\tfrac{\sqrt{n}}{2}$ for $k=0,1,\ldots$ whenever $\|\Delta\|_{\rm op}\le\tfrac{n}{16}$.  Thus, the error bound in Proposition~\ref{prop:err-bd} applies to the entire sequence $\{z^k\}_{k\ge0}$. Third, since every global maximizer $\hat{z}$ of Problem~\eqref{opt:QP} is a second-order critical point, we have $\rho(\hat{z})=0$ by Fact~\ref{lem:fix-pt}.  Proposition~\ref{prop:err-bd} shows that the converse is also true.  Hence, we can view $\rho$ as a surrogate measure of optimality and use it to keep track of Algorithm~\ref{alg:GP}'s progress.

\begin{proof}[Proof of Proposition~\ref{prop:err-bd}]
Since 
\begin{equation}\label{ineq:pf-5}
\rho(z) = \|\Sigma(z)z\|_2 \geq \|\Sigma(\hat{z})z\|_2 - \|(\Sigma(z)-\Sigma(\hat{z}))z\|_2,
\end{equation}
it suffices to establish an upper bound on $\|(\Sigma(z)-\Sigma(\hat{z}))z\|_2$ and a lower bound on $\|\Sigma(\hat{z})z\|_2$.  Towards that end, recall that
$$ \tilde{C}=C+\frac{n}{\alpha}I = (z^\star)(z^\star)^H+\Delta + \frac{n}{\alpha}I $$
and let
$$ \hat{\theta}=\arg\min_{\theta\in[0,2\pi)} \|z-e^{i\theta}\hat{z}\|_2, \quad \hat{\theta}^\star = \arg\min_{\theta\in[0,2\pi)} \|\hat{z}-e^{i\theta}z^\star\|_2. $$
First, we bound
\begin{eqnarray}
  & & \|(\Sigma(z)-\Sigma(\hat{z}))z\|_2 \nonumber \\
  \noalign{\smallskip}
  &=& \left\|\left(\text{Diag}(|\tilde{C}z|)-\text{Diag}(|\tilde{C}\hat{z}|)\right)z\right\|_2 \nonumber \\
\noalign{\smallskip}
&=& \left( \sum_{j=1}^n \left| \left( |(\tilde{C}z)_j| - |(\tilde{C}\hat{z})_j| \right)z_j \right|^2 \right)^{1/2} \nonumber \\
\noalign{\smallskip}
&=& \left\||\tilde{C}e^{-i\hat{\theta}}z|-|\tilde{C}\hat{z}|\right\|_2 \nonumber \\
\noalign{\smallskip}
&\le& \|\tilde{C}(e^{-i\hat{\theta}}z-\hat{z})\|_2 \nonumber\\
\noalign{\smallskip}
&\le& \|(z^\star)(z^\star)^H(e^{-i\hat{\theta}}z-\hat{z})\|_2 + \|\Delta(e^{-i\hat{\theta}}z-\hat{z})\|_2 + \frac{n}{\alpha}\|e^{-i\hat{\theta}}z-\hat{z}\|_2 \nonumber \\
\noalign{\smallskip}
&\le& \sqrt{n} \cdot | (z^\star)^H(e^{-i\hat{\theta}}z-\hat{z}) | + \left( \|\Delta\|_{\rm op} + \frac{n}{\alpha} \right) d_2(z,\hat{z}). \label{ineq:pf-1}
\end{eqnarray}
By definition of $\hat{\theta}$, we have $\hat{z}^H(e^{-i\hat{\theta}}z) = |\hat{z}^Hz|$, which implies that
\begin{equation} \label{eq:norm-hat}
  \| e^{-i\hat{\theta}}z-\hat{z} \|_2^2 = 2 (n - |\hat{z}^Hz|).
\end{equation}
This, together with Fact~\ref{lem:1}, yields
\begin{eqnarray}
  | (z^\star)^H(e^{-i\hat{\theta}}z-\hat{z}) | &\le& | (z^\star-e^{-i\hat{\theta}^\star}\hat{z})^H(e^{-i\hat{\theta}}z-\hat{z}) | + | (e^{-i\hat{\theta}^\star}\hat{z})^H(e^{-i\hat{\theta}}z-\hat{z}) | \nonumber \\
  \noalign{\smallskip}
  &\le& \| z^\star-e^{-i\hat{\theta}^\star}\hat{z} \|_2 \cdot \| e^{-i\hat{\theta}}z-\hat{z} \|_2 + \left| |\hat{z}^Hz|-n \right| \nonumber \\
  \noalign{\smallskip}
  &\le& \frac{4\|\Delta\|_{\rm op}}{\sqrt{n}} \cdot d_2(z,\hat{z}) + \frac{1}{2} d_2(z,\hat{z})^2. \label{ineq:pf-2}
\end{eqnarray}
Upon substituting~\eqref{ineq:pf-2} into~\eqref{ineq:pf-1}, we obtain
\begin{equation} \label{eq:ub}
\|(\Sigma(z)-\Sigma(\hat{z}))z\|_2 \le \left(5\|\Delta\|_{\rm op} + \frac{n}{\alpha} \right) d_2(z,\hat{z}) + \frac{\sqrt{n}}{2} d_2(z,\hat{z})^2. 
\end{equation}

Next, let $\hat{u}=\left( I-\tfrac{1}{n}\hat{z}\hat{z}^H \right)(e^{-i\hat{\theta}}z-\hat{z})$ be the projection of $e^{-i\hat{\theta}}z-\hat{z}$ onto the orthogonal complement of $\mbox{span}(\hat{z})$.  Hence, we have $\hat{u}^H\hat{z}=0$ and
\begin{equation}  \label{ineq:pf-8}
  \|\hat{u}\|_2 \geq \left\| e^{-i\hat{\theta}}z-\hat{z} \right\|_2 - \left\|\frac{1}{n}\hat{z}\hat{z}^H(e^{-i\hat{\theta}}z-\hat{z})\right\|_2 = d_2(z,\hat{z}) - \frac{1}{2\sqrt{n}}d_2(z,\hat{z})^2,
\end{equation}
where the last equality follows from~\eqref{eq:norm-hat}.  Moreover, by definition of $\hat{\theta}^\star$, we have $(z^\star)^H(e^{-i\hat{\theta}^\star}\hat{z})=|(z^\star)^H\hat{z}|$ and
\begin{equation} \label{eq:star-hat}
  |(z^\star)^H\hat{z}| = n-\frac{1}{2}\|\hat{z}-e^{i\hat{\theta}^\star}z^\star\|_2^2.
\end{equation}
Hence,
\begin{eqnarray}
  \hat{u}^H\Sigma(\hat{z})\hat{u} &=& \hat{u}^H\left( \mbox{Diag}(|\tilde{C}\hat{z}|)-\tilde{C} \right)\hat{u} \nonumber \\
  \noalign{\smallskip}
  &=& \hat{u}^H\left( \mbox{Diag}(|C\hat{z}|)-C \right)\hat{u} \label{eq:Ctilde} \\
  \noalign{\smallskip}
  &\ge& \left(\sum_{j=1}^n |(C\hat{z})_j|\cdot |\hat{u}_j|^2\right) - |(z^\star)^H\hat{u}|^2 - \hat{u}^H\Delta\hat{u} \nonumber \\
  \noalign{\smallskip}
  &\ge& \left(|(z^\star)^H\hat{z}| - \|\Delta\hat{z}\|_\infty \right) \|\hat{u}\|_2^2 - \left| \hat{u}^H(z^\star-e^{-i\hat{\theta}^\star}\hat{z}) \right|^2 - \|\Delta\|_{\rm op}\|\hat{u}\|_2^2 \nonumber \\
  \noalign{\smallskip}
  &\ge&  \left( n - \|\Delta\hat{z}\|_\infty  - \|\Delta\|_{\rm op} - \frac{3}{2}\|z^\star-e^{-i\hat{\theta}^\star}\hat{z}\|_2^2 \right) \|\hat{u}\|_2^2 \label{eq:star-hat-2} \\
  \noalign{\smallskip}
  &\ge& \left( n - \|\Delta\hat{z}\|_\infty  - \|\Delta\|_{\rm op} - \frac{24\|\Delta\|_{\rm op}^2}{n} \right) \|\hat{u}\|_2^2, \label{ineq:pf-9}
\end{eqnarray}
where~\eqref{eq:Ctilde} follows from Fact~\ref{lem:fix-pt} and the fact that $\hat{z}$ is a second-order critical point of Problem~\eqref{opt:QP},~\eqref{eq:star-hat-2} is due to~\eqref{eq:star-hat}, and~\eqref{ineq:pf-9} follows from Fact~\ref{lem:1}.  Since $\Sigma(\hat{z})\hat{z}=\bz$ by Fact~\ref{lem:fix-pt}, we obtain from~\eqref{ineq:pf-8} and~\eqref{ineq:pf-9} that
\begin{eqnarray}
  & & \|\Sigma(\hat{z})z\|_2 \,\,\,=\,\,\, \|\Sigma(\hat{z})\hat{u}\|_2 \nonumber \\
  \noalign{\smallskip}
  &\ge& \left( n - \|\Delta\hat{z}\|_\infty  - \|\Delta\|_{\rm op} - \frac{24\|\Delta\|_{\rm op}^2}{n} \right) \left( d_2(z,\hat{z}) - \frac{1}{2\sqrt{n}}d_2(z,\hat{z})^2 \right). \label{ineq:pf-3}
\end{eqnarray}

Now, by Fact~\ref{lem:1} and the assumption that $d_2(z,z^\star)\le\tfrac{\sqrt{n}}{2}$, we have
$$
  d_2(z,\hat{z}) \le d_2(z,z^\star) + d_2(\hat{z},z^\star) \le \frac{\sqrt{n}}{2} + \frac{4\|\Delta\|_{\rm op}}{\sqrt{n}}.
$$
This implies that
\begin{equation} \label{ineq:pf-11}
  d_2(z,\hat{z})^2 \le \left( \frac{\sqrt{n}}{2} + \frac{4\|\Delta\|_{\rm op}}{\sqrt{n}} \right) d_2(z,\hat{z}). 
\end{equation}
Moreover,
\begin{eqnarray}
  \|\Delta\hat{z}\|_\infty &\le& \|\Delta z^\star\|_\infty + \|\Delta(e^{-i\hat{\theta}^\star}\hat{z}-z^\star)\|_\infty \nonumber \\
  \noalign{\smallskip}
  &\le& \|\Delta z^\star\|_\infty + \|\Delta\|_{\rm op}\cdot d_2(\hat{z},z^\star) \nonumber \\
  \noalign{\smallskip}
  &\le& \|\Delta z^\star\|_\infty + \frac{4\|\Delta\|_{\rm op}^2}{\sqrt{n}}. \label{eq:infty-bd}
\end{eqnarray}
It follows from~\eqref{ineq:pf-5},~\eqref{eq:ub},~\eqref{ineq:pf-3},~\eqref{ineq:pf-11}, and~\eqref{eq:infty-bd} that
\begin{eqnarray*}
  \rho(z) &\ge& \left[ \left( \frac{1}{2}-\frac{1}{\alpha} \right)n - \frac{3\|\Delta\hat{z}\|_\infty}{4} - \frac{39\|\Delta\|_{\rm op}}{4}  - \frac{16\|\Delta\|_{\rm op}^2}{n} \right] d_2(z,\hat{z}) \\
  \noalign{\smallskip}
  &\ge& \left[ \frac{n}{4} - \frac{3\|\Delta z^\star\|_\infty}{4} - \frac{39\|\Delta\|_{\rm op}}{4} - \frac{3\|\Delta\|_{\rm op}^2}{\sqrt{n}} - \frac{16\|\Delta\|_{\rm op}^2}{n} \right] d_2(z,\hat{z}) \\
  \noalign{\smallskip}
  &\ge& \frac{n}{8}d_2(z,\hat{z})
\end{eqnarray*}
whenever $\|\Delta\|_{\rm op} \le \tfrac{n^{3/4}}{312}$, $\|\Delta z^\star\|_\infty \le \tfrac{n}{24}$, and $\alpha\ge4$.  This completes the proof.
\end{proof} 

We note that under a slightly more restrictive noise setting, one can establish an error bound similar to that in Proposition~\ref{prop:err-bd} to estimate the distance between any point in a neighborhood of $z^\star$ and the set of \emph{second-order critical points} of Problem~\eqref{opt:QP}.  Specifically, we have the following result:
\begin{proposition} \label{prop:EB-2crit}
Let $\Sigma:\T^n\rightarrow\mathbb{H}^n$ and $\rho:\T^n\rightarrow\R_+$ be as in Proposition~\ref{prop:err-bd}.  Suppose that (i) the measurement noise $\Delta$ satisfies $\|\Delta\|_{\rm op}\le\tfrac{n^{2/3}}{32768}$ and $\|\Delta z^\star\|_\infty \le \tfrac{n}{24}$, and (ii) the parameter $\alpha$ satisfies $\alpha\ge4$.  Then, for any point $z\in\T^n$ satisfying $d_2(z,z^\star)\le\tfrac{\sqrt{n}}{2}$ and any second-order critical point $\tilde{z}\in\T^n$ of Problem~\eqref{opt:QP}, we have
$$ d_2(z,\tilde{z}) \le \frac{8}{n}\rho(z). $$
\end{proposition}
\begin{proof}
  The proof is essentially the same as that of Proposition~\ref{prop:err-bd}.  To save space, let us just highlight the key steps.  Similar to~\eqref{ineq:pf-5}, we have
  \begin{equation} \label{eq:2crit-0}
    \rho(z) = \|\Sigma(z)z\|_2 \geq \|\Sigma(\tilde{z})z\|_2 - \|(\Sigma(z)-\Sigma(\tilde{z}))z\|_2.
  \end{equation}
  Define
  $$ \tilde{\theta}=\arg\min_{\theta\in[0,2\pi)} \|z-e^{i\theta}\tilde{z}\|_2, \quad \tilde{\theta}^\star = \arg\min_{\theta\in[0,2\pi)} \|\tilde{z}-e^{i\theta}z^\star\|_2. $$
  Then, since $d_2(\tilde{z},z^\star) =  \|z^\star-e^{-i\tilde{\theta}^\star}\tilde{z}\|_2 \le \sqrt{8\|\Delta\|_{\rm op}}$ by Fact~\ref{lem:fix-pt}, we have
  \begin{eqnarray}
    & & \|(\Sigma(z)-\Sigma(\tilde{z}))z\|_2 \nonumber \\
    \noalign{\smallskip}
    &\le& \sqrt{n} \cdot | (z^\star)^H(e^{-i\tilde{\theta}}z-\tilde{z}) | + \left( \|\Delta\|_{\rm op} + \frac{n}{\alpha} \right) d_2(z,\tilde{z}) \nonumber \\
    \noalign{\smallskip}
    &\le& \sqrt{n}\left( \|z^\star-e^{-i\tilde{\theta}^\star}\tilde{z}\|_2\cdot\|e^{-i\tilde{\theta}}z-\tilde{z}\|_2 + \left| |\tilde{z}^Hz| - n \right| \right) \nonumber \\
    \noalign{\smallskip}
    && \,\,+\,\, \left( \|\Delta\|_{\rm op} + \frac{n}{\alpha} \right) d_2(z,\tilde{z}) \nonumber \\
    \noalign{\smallskip}
    &\le& \left( \sqrt{8n\|\Delta\|_{\rm op}} + \|\Delta\|_{\rm op} + \frac{n}{\alpha} \right) d_2(z,\tilde{z}) + \frac{\sqrt{n}}{2}d_2(z,\tilde{z})^2. \label{eq:2crit-1}
  \end{eqnarray}
  
  Now, let $\tilde{u}=\left( I-\tfrac{1}{n}\tilde{z}\tilde{z}^H \right)(e^{-i\tilde{\theta}}z-\tilde{z})$ be the projection of $e^{-i\tilde{\theta}}z-\tilde{z}$ onto the orthogonal complement of $\mbox{span}(\tilde{z})$.  Then, similar to the derivation of~\eqref{ineq:pf-3}, we have
  \begin{eqnarray}
  & & \|\Sigma(\tilde{z})z\|_2 \,\,\,=\,\,\, \|\Sigma(\tilde{z})\tilde{u}\|_2 \nonumber \\
  \noalign{\smallskip}
  &\ge& \left(|(z^\star)^H\tilde{z}| - \|\Delta\tilde{z}\|_\infty \right) \|\tilde{u}\|_2^2 - \left| \tilde{u}^H(z^\star-e^{-i\tilde{\theta}^\star}\tilde{z}) \right|^2 - \|\Delta\|_{\rm op}\|\tilde{u}\|_2^2 \nonumber \\
  \noalign{\smallskip}
  &\ge& \left( n - \|\Delta\tilde{z}\|_\infty  - 13\|\Delta\|_{\rm op} \right) \left( d_2(z,\tilde{z}) - \frac{1}{2\sqrt{n}}d_2(z,\tilde{z})^2 \right). \label{eq:2crit-2}
  \end{eqnarray}
  Moreover, following the derivations of~\eqref{ineq:pf-11} and~\eqref{eq:infty-bd}, we have
  \begin{eqnarray}
    d_2(z,\tilde{z}) &\le& \frac{\sqrt{n}}{2} + \sqrt{8\|\Delta\|_{\rm op}}, \label{eq:2crit-3} \\
    \noalign{\smallskip}
    \|\Delta\tilde{z}\|_\infty &\le& \|\Delta z^\star\|_\infty + \sqrt{8}\|\Delta\|_{\rm op}^{3/2}.  \label{eq:2crit-4}
  \end{eqnarray}
  Upon putting together~\eqref{eq:2crit-0}--\eqref{eq:2crit-4}, we obtain
  \begin{eqnarray*}
    \rho(z) &\ge& \left[ \frac{n}{4} - \frac{3\|\Delta z^\star\|_\infty}{4} - 4\sqrt{2n\|\Delta\|_{\rm op}} - \frac{43\|\Delta\|_{\rm op}}{4} - \frac{3\sqrt{2}\|\Delta\|_{\rm op}^{3/2}}{2} \right] d_2(z,\tilde{z}) \\
    \noalign{\smallskip}
    &\ge& \frac{n}{8}d_2(z,\tilde{z})
  \end{eqnarray*}
  whenever $\|\Delta\|_{\rm op}\le\tfrac{n^{2/3}}{32768}$, $\|\Delta z^\star\|_\infty\le\tfrac{n}{24}$, and $\alpha\ge4$.  This completes the proof.
\end{proof}


Recall that a global maximizer of Problem~\eqref{opt:QP} is a second-order critical point.  Now, under the assumptions of Proposition~\ref{prop:EB-2crit}, we know that every second-order critical point $\tilde{z}$ of Problem~\eqref{opt:QP} satisfies $d_2(\tilde{z},z^\star) \le \sqrt{8\|\Delta\|_{\rm op}} \le \tfrac{\sqrt{n}}{2}$; see Fact~\ref{lem:fix-pt}. Thus, Proposition~\ref{prop:EB-2crit} shows that every second-order critical point of Problem~\eqref{opt:QP} is also a global maximizer, which is unique up to a global phase.  This gives an alternative proof of~\cite[Theorem 4]{boumal2016nonconvex} with a less restrictive requirement on $\|\Delta z^\star\|_\infty$ ($\|\Delta z^\star\|_\infty = O(n)$ in Proposition~\ref{prop:EB-2crit} vs. $\|\Delta z^\star\|_\infty = O(n^{2/3}\sqrt{\log n})$ in~\cite[Theorem 4]{boumal2016nonconvex}).  It remains an open question to determine whether the conclusion of Proposition~\ref{prop:EB-2crit} still holds under the same noise requirement as Proposition~\ref{prop:err-bd}.

Now, let us proceed to the second part of the proof of Theorem~\ref{thm:convergence-rate}.  Our goal is to prove the following proposition, which elucidates the key properties of Algorithm~\ref{alg:GP}:
\begin{proposition} \label{prop:alg-prop}
  Under the assumptions of Theorem~\ref{thm:convergence-rate}, the sequence of iterates $\{z^k\}_{k\ge0}$ generated by Algorithm~\ref{alg:GP} satisfies the following for $k=0,1,\ldots$, where $a_0,a_1,a_2>0$ are quantities that depend only on $n$ and $\alpha$, and $\hat{z}$ is any global maximizer of Problem~\eqref{opt:QP}:
  \begin{enumerate}
  \item[(a)] \emph{(Sufficient Ascent)} $f(z^{k+1})-f(z^k) \ge a_0\cdot\|z^{k+1}-z^k\|_2^2$.

  \item[(b)] \emph{(Cost-to-Go Estimate)} $f(\hat{z}) - f(z^k) \le a_1 \cdot d_2(z^k,\hat{z})^2$.

  \item[(c)] \emph{(Safeguard)} $\rho(z^k) \le a_2\cdot \|z^{k+1}-z^k\|_2$.    
  \end{enumerate}
\end{proposition}
\begin{proof}
  We begin by proving (a).  Recalling that $\tilde{C}=C+\tfrac{n}{\alpha}I$, we have
  \begin{eqnarray*}
    & & f(z^{k+1})-f(z^k) \\
    \noalign{\smallskip}
    &=& (z^{k+1}-z^k)^H\tilde{C}(z^{k+1}-z^k) - 2(z^k)^H\tilde{C}(z^k) + 2\Re\{ (z^{k+1})^H\tilde{C}(z^k) \}.
  \end{eqnarray*}
  We claim that $\Re\{ (z^{k+1})^H\tilde{C}(z^k) \} \geq (z^k)^H \tilde{C} (z^k)$.  This follows from the fact that
  $$ (z^{k+1})^H\tilde{C}(z^k) = \left( \frac{\tilde{C}z^k}{|\tilde{C}z^k|} \right)^H\tilde{C}(z^k) $$
  is a real number and
  $$ z^{k+1}\in \arg\max_{z\in \mathbb{T}^n} \Re\{z^H\tilde{C}z^k\}. $$
  Hence, by the assumption on $\alpha$, we have
  $$ f(z^{k+1})-f(z^k) \ge (z^{k+1}-z^k)^H\tilde{C}(z^{k+1}-z^k) \ge a_0 \cdot \|z^{k+1}-z^k\|_2^2 $$
  with $a_0=\lambda_{\rm min}\left(\Delta+\frac{n}{\alpha}I\right)>0$.
  
  Next, we prove (b).  Let $\hat{\theta}_k=\arg\min_{\theta\in[0,2\pi)} \|z^k-e^{i\theta}\hat{z}\|_2$.  Then, we have
  \begin{eqnarray}
    f(\hat{z}) - f(z^k) &=& \hat{z}^H\tilde{C}\hat{z} - (z^k)^H\tilde{C}(z^k) \nonumber \\
    \noalign{\smallskip}
    &=& \|\tilde{C}\hat{z}\|_1 - (z^k)^H\tilde{C}(z^k) \label{eq:ctg-1} \\
    \noalign{\smallskip}
    &=& (z^k)^H \left( {\rm Diag}(|\tilde{C}\hat{z}|) - \tilde{C} \right) (z^k) \nonumber \\
    \noalign{\smallskip}
    &=& (e^{-i\hat{\theta}_k}z^k-\hat{z})^H \left( {\rm Diag}(|\tilde{C}\hat{z}|) - \tilde{C} \right) (e^{-i\hat{\theta}_k}z^k-\hat{z}) \label{eq:ctg-2} \\
    \noalign{\smallskip}
    &\le& \left( \|\tilde{C}\|_{\rm op} + \|\tilde{C}\hat{z}\|_\infty \right) d_2(z^k,\hat{z})^2, \nonumber
  \end{eqnarray}
  where both~\eqref{eq:ctg-1} and~\eqref{eq:ctg-2} follow from Fact~\ref{lem:fix-pt}.  Now, observe that
  \begin{eqnarray}
    & & \|\tilde{C}\|_{\rm op} + \|\tilde{C}\hat{z}\|_\infty \nonumber \\
    \noalign{\smallskip}
    &\le& \|C\|_{\rm op}+\|C\hat{z}\|_\infty + \frac{2n}{\alpha} \nonumber \\
    \noalign{\smallskip}
    &\le& \|(z^\star)(z^\star)^H\|_{\rm op} + \|\Delta\|_{\rm op} + \|(z^\star)(z^\star)^H\hat{z}\|_\infty + \|\Delta\hat{z}\|_\infty + \frac{2n}{\alpha} \nonumber \\
    \noalign{\smallskip}
    &\le& n+\|\Delta\|_{\rm op} + |(z^\star)^H\hat{z}| + \|\Delta z^\star\|_\infty + \frac{4\|\Delta\|_{\rm op}^2}{\sqrt{n}} + \frac{2n}{\alpha} \label{eq:ctg-3} \\
    \noalign{\smallskip}
    &\le& 2n + \|\Delta\|_{\rm op} + \|\Delta z^\star\|_\infty + \frac{4\|\Delta\|_{\rm op}^2}{\sqrt{n}} + \frac{2n}{\alpha} \label{eq:ctg-4} \\
    \noalign{\smallskip}
    &<& 3n, \label{eq:ctg-5} 
  \end{eqnarray}
  where~\eqref{eq:ctg-3} follows from~\eqref{eq:infty-bd} and the fact that $\|(z^\star)(z^\star)^H\|_{\rm op}=n$,~\eqref{eq:ctg-4} follows from~\eqref{eq:star-hat} and Fact~\ref{lem:1}, and~\eqref{eq:ctg-5} is due to the assumptions on $\alpha$, $\|\Delta\|_{\rm op}$, and $\|\Delta z^\star\|_\infty$.  Hence, we conclude that
  $$ f(\hat{z})-f(z^k)\le a_1\cdot d_2(z^k,\hat{z})^2 $$
  for some $a_1\in(0,3n)$.

  Lastly, we prove (c).  By definition of $z^{k+1}$, we have
  $$ {\rm Diag}(|\tilde{C}z^k|)(z^{k+1}-z^k) = \left( \tilde{C}-{\rm Diag}(|\tilde{C}z^k|) \right)z^k. $$
  It follows that
  $$ \rho(z^k) = \left\| {\rm Diag}(|\tilde{C}z^k|)(z^{k+1}-z^k) \right\|_2 \le \|\tilde{C}z^k\|_\infty\|z^{k+1}-z^k\|_2. $$
  Now, recall from the proof of Theorem \ref{thm:2-norm} that $d_2(z^k,z^\star)\leq \tfrac{\sqrt{n}}{2}$ for $k=0,1,\ldots$.  Upon letting $\theta_k^\star=\arg\min_{\theta\in[0,2\pi)} \|z^k-e^{i\theta}z^\star\|_2$, we obtain
  \begin{eqnarray*}
    \|\tilde{C}z^k\|_\infty &\le& \|(z^\star)(z^\star)^Hz^k\|_\infty + \|\Delta z^k\|_\infty + \frac{n}{\alpha} \\
    \noalign{\smallskip}
    &\le& |(z^\star)^Hz^k| + \|\Delta(e^{-i\theta_k^\star}z^k-z^\star)\|_\infty + \|\Delta z^\star\|_\infty + \frac{n}{\alpha} \\
    \noalign{\smallskip}
    &\le& n + \|\Delta\|_{\rm op}\cdot d_2(z^k,z^\star) + \|\Delta z^\star\|_\infty + \frac{n}{\alpha} \\
    \noalign{\smallskip}
    &<& \frac{3n^{5/4}}{2},
  \end{eqnarray*}
  where the last inequality is due to the assumptions on $\alpha$, $\|\Delta\|_{\rm op}$, and $\|\Delta z^\star\|_\infty$.  It follows that
  $$ \rho(z^k) \le a_2 \cdot \|z^{k+1}-z^k\|_2  $$
  for some $a_2\in\left( 0,\tfrac{3n^{5/4}}{2} \right)$.
\end{proof}

Armed with Propositions~\ref{prop:err-bd} and~\ref{prop:alg-prop}, we are now ready to prove Theorem~\ref{thm:convergence-rate}.
\begin{proof} [Proof of Theorem~\ref{thm:convergence-rate}]
  By Propositions~\ref{prop:err-bd} and~\ref{prop:alg-prop}, we have
  \begin{eqnarray*}
    f(\hat{z}) - f(z^{k+1}) &=& \left( f(\hat{z}) - f(z^k) \right) - \left( f(z^{k+1}) - f(z^k) \right) \nonumber \\
    \noalign{\smallskip}
    &\le& a_1\cdot d_2(z^k,\hat{z})^2 - \left( f(z^{k+1}) - f(z^k) \right) \nonumber \\
    \noalign{\smallskip}
    &\le& \frac{64a_1}{n^2} \rho(z^k)^2 - \left( f(z^{k+1}) - f(z^k) \right) \nonumber \\
    \noalign{\smallskip}
    &\le& \frac{64a_1a_2^2}{n^2}\|z^{k+1}-z^k\|_2^2 - \left( f(z^{k+1}) - f(z^k) \right) \nonumber \\
    \noalign{\smallskip}
    &\le& \left( \frac{64a_1a_2^2}{a_0n^2} - 1 \right) \left( f(z^{k+1}) - f(\hat{z}) + f(\hat{z}) - f(z^k) \right).
  \end{eqnarray*}
  Since $f(\hat{z})\ge f(z^k)$ for $k=0,1,\ldots$, we may assume without loss of generality that $a'=\tfrac{64a_1a_2^2}{a_0n^2}>1$.  It then follows that
  $$ f(\hat{z})-f(z^{k+1}) \le \frac{a'-1}{a'}\left( f(\hat{z}) - f(z^k) \right), $$
  which yields
  $$
    f(\hat{z})-f(z^k) \le \left( f(\hat{z})-f(z^0) \right) \lambda^k
  $$
  with $\lambda=\tfrac{a'-1}{a'}\in(0,1)$.  Furthermore, we have
  \begin{eqnarray*}
    d_2(z^k,\hat{z})^2 &\le& \frac{64}{n^2}\rho(z^k)^2 \\
    \noalign{\smallskip}
    &\le& \frac{64a_2^2}{n^2} \|z^{k+1}-z^k\|_2^2 \\
    \noalign{\smallskip}
    &\le& \frac{64a_2^2}{a_0n^2}\left( f(z^{k+1})-f(z^k) \right) \\
    \noalign{\smallskip}
    &\le& \frac{64a_2^2}{a_0n^2}\left( f(\hat{z})-f(z^k) \right) \\
    \noalign{\smallskip}
    &\le& \frac{64a_2^2}{a_0n^2} \left( f(\hat{z})-f(z^0) \right) \lambda^k,
  \end{eqnarray*}
  which implies that
  $$ d_2(z^k,\hat{z}) \le a \left( f(\hat{z})-f(z^0) \right)^{1/2} \lambda^{k/2} $$
  with $a=\sqrt{\tfrac{64a_2^2}{a_0n^2}}$. This completes the proof.
\end{proof}

Again, we can specialize Theorem~\ref{thm:convergence-rate} to the Gaussian noise setting.  This leads to the following corollary, which can be proven by combining Theorem~\ref{thm:convergence-rate} with the probabilistic estimates in~\cite[Proposition 3.3]{BBS16}; cf.~Corollary~\ref{cor:1}:
\begin{corollary} \label{cor:global-conv}
Suppose that the measurement noise $\Delta$ takes the form $\Delta=\sigma W$, where $\sigma^2>0$ is the noise power satisfying $\sigma\in\left(0,\tfrac{n^{1/4}}{936}\right]$ and $W\in\mathbb{H}^n$ is a Wigner matrix.  Suppose further that the step size $\alpha$ satisfies $\alpha\in\left[ 4,312n^{1/4} \right)$ and the initial point $z^0$ is given by $z^0=v_C$.  Then, with probability at least $1-2n^{-5/4}-2e^{-n/2}$, the sequence of iterates $\{z^k\}_{k\ge0}$ generated by Algorithm~\ref{alg:GP} satisfies
  \begin{eqnarray*}
    f(\hat{z}) - f(z^k) &\le& \left( f(\hat{z})-f(z^0) \right) \lambda^k, \\
    \noalign{\smallskip}
    d_2(z^k,\hat{z}) &\leq& a \left( f(\hat{z})-f(z^0) \right)^{1/2} \lambda^{k/2}
  \end{eqnarray*}
  for $k=0,1,\ldots$, where $a>0,\lambda\in(0,1)$ are quantities that depend only on $n$ and $\alpha$, and $\hat{z}$ is any global maximizer of Problem~\eqref{opt:QP}.
\end{corollary}


Corollary~\ref{cor:global-conv} shows that in the Gaussian noise setting, Algorithm~\ref{alg:GP} will converge to a global maximizer of Problem~\eqref{opt:QP} at a linear rate with high probability for noise level up to $\sigma=O(n^{1/4})$.  This matches the noise level requirement for the tightness of the SDR-based method established in~\cite[Theorem 2.1]{BBS16}.  As the GPM typically has lower complexity than the SDR-based method in tackling Problem~\eqref{opt:QP}, we see that the former is competitive with the latter in terms of both theoretical guarantees and numerical efficiency.

\section{Conclusion}
In this paper, we conducted a comprehensive analysis of the estimation and convergence performance of the GPM for tackling the phase synchronization problem.  First, under the assumption that the measurement noise $\Delta$ satisfies $\|\Delta\|_{\rm op}=O(n)$, we established bounds on the rates of decrease in the $\ell_2$- and $\ell_\infty$-estimation errors of the iterates generated by the GPM.  As a corollary, we showed that in the Gaussian noise setting (i.e., $\Delta=\sigma W$, where $\sigma>0$ is the noise level and $W$ is a Wigner matrix), the expected squared $\ell_2$-estimation errors of the iterates are decreasing and all are on the same order as that of the MLE even when the noise level is $\sigma=O(n^{1/2})$.  The above result holds regardless of whether the iterates converge or not and yields the best provable bound on the estimation error of any accumulation point generated by the GPM under the least restrictive noise requirement currently known.  Second, we showed that when the measurement noise $\Delta$ and target phase vector $z^\star$ satisfy $\|\Delta\|_{\rm op}=O(n^{3/4})$ and $\|\Delta z^\star\|_\infty=O(n)$, the GPM will converge linearly to a global maximizer of Problem~\eqref{opt:QP}.  This not only resolves an open question in~\cite{boumal2016nonconvex} concerning the convergence \emph{rate} of the GPM but also improves upon the noise requirement $\|\Delta\|_{\rm op}=O(n^{2/3})$ and $\|\Delta z^*\|_\infty=O(n^{2/3}\sqrt{\log n})$ that is imposed in~\cite{boumal2016nonconvex} to establish just the convergence of the GPM.  Our result implies that in the Gaussian noise setting, the GPM will converge linearly to a global maximizer of Problem~\eqref{opt:QP} in the noise regime $\sigma=O(n^{1/4})$.  This is the same regime for which the computationally heavier SDR-based method in~\cite{BBS16} is provably tight.  To establish our convergence rate result, we developed a new error bound for the non--convex problem~\eqref{opt:QP}.  As a by-product, we showed that every second-order critical point of Problem~\eqref{opt:QP} is globally optimal if $\|\Delta\|_{\rm op}=O(n^{2/3})$ and $\|\Delta z^*\|_\infty=O(n)$.  This slightly improves upon the corresponding result in~\cite{boumal2016nonconvex}.  An interesting future direction would be to extend the GPM and the machinery developed in this paper to design and analyze first-order methods for other (non-convex) quadratic optimization problems.


\section*{Acknowledgement}
We thank Nicolas Boumal for his helpful comments on an earlier version of our manuscript.

\bibliographystyle{plain}
\bibliography{as}

\end{document}